\documentclass[12pt]{article}
\usepackage{latexsym,amsfonts,amsmath,graphics,amsthm}
\usepackage{epsfig}
\usepackage{mathrsfs}

\newtheorem{theorem}{Theorem}

\newtheorem{corollary}{Corollary}

\newtheorem{definition}{Definition}

\newtheorem{remark}{Remark}

\newcommand{\bfA}{\mathbf{A}} 
\newcommand{\bfB}{\mathbf{B}} 
\newcommand{\bfa}{\mathbf{a}}
\newcommand{\bfk}{\mathbf{k}}
\newcommand{\Tfett}{{\bf T}}
\newcommand{\bsalpha}{\boldsymbol{\alpha}}

\newcommand{\bsbeta}{\boldsymbol{\beta}}

\newcommand{\bsx}{\boldsymbol{x}}

\newcommand{\bsy}{\boldsymbol{y}}

\newcommand{\Ccal}{{\cal C}}
\newcommand{\cN}{{\cal N}}

\newcommand{\bszero}{\boldsymbol{0}}

\newcommand{\NN}{\mathbb{N}}
\newcommand{\ZZ}{\mathbb{Z}}
\newcommand{\RR}{\mathbb{R}}

\newcommand{\cS}{\mathcal S}
\newcommand{\FF}{\mathbb{F}}

\newcommand{\cP}{{\mathscr P}}

\makeatletter
\newcommand{\rdots}{\mathinner{\mkern1mu\lower-1\p@\vbox{\kern7\p@\hbox{.}}
\mkern2mu \raise4\p@\hbox{.}\mkern2mu\raise7\p@\hbox{.}\mkern1mu}}
\makeatother

\allowdisplaybreaks

\begin{document}

\title{On the existence of hyperplane sequences, \\with quality parameter and
discrepancy bounds}

\author{Friedrich Pillichshammer and Gottlieb Pirsic
\thanks{The research of G.Pirsic was supported by the Austrian Science Fund, Project P23285-N18}
}

\date{}
\maketitle

\begin{abstract}
It is well-known that digital $(t,m,s)$-nets and $(\Tfett,s)$-sequences over a finite field have excellent properties when they are used as underlying nodes in quasi-Monte Carlo integration rules. One very general sub-class of digital nets are hyperplane nets which can be viewed as a generalization of cyclic nets and of polynomial lattice point sets. In this paper we introduce infinite versions of hyperplane nets and call these sequences hyperplane sequences. Our construction is based on the recent duality theory for digital sequences according to Dick and Niederreiter. We then analyze the equidistribution properties of hyperplane sequences in terms of the quality function $\Tfett$ and the star discrepancy.
\end{abstract}

\section{Introduction}

In the field of quasi-Monte Carlo (QMC) integration one approximates the integral of a function over the (usually high-dimensional) unit-cube by the average of function evaluations over a well-chosen deterministic point set. Such an integration rule is called a QMC-algorithm. The fundamental error estimate for this method is the Koksma-Hlawka inequality which states that the absolute error of a QMC-algorithm is bounded by the variation in the sense of Hardy and Krause of the integrand times the star discrepancy of the underlying point set. Hence, point sets from the unit-cube having a very uniform distribution, or in other words, having low star discrepancy yield small integration errors, at least for functions with bounded variation in the sense of Hardy and Krause. For more information on QMC we refer to the books \cite{DP2010,niesiam}. 

The star discrepancy of a point set $\cP=\{\bsx_0,\ldots,\bsx_{N-1}\}$ consisting of $N$ points in $[0,1)^s$ is defined as $$D_N^{\ast}(\cP)=\sup_{J}\left|\frac{A(J,N)}{N}-\lambda_s(J)\right|,$$ where the supremum is extended over all intervals of the form $$J=[0,a_1)\times \ldots \times [0,a_s) \subseteq [0,1)^s,$$ where $A(J,N)$ is the number of indices $n$ for which $\bsx_n$ belongs to $J$ and $\lambda_s(J)$ is the $s$-dimensional volume of $J$. For an infinite sequence $\cS=(\bsy_{n})_{n \ge 0}$ for $N \in \NN$ the star discrepancy $D_N^{\ast}(\cS)$ is the star discrepancy of the point set consisting of the first $N$ elements of $\cS$.

Currently the best constructions of point sets with low star discrepancy are based on the concept of $(t,m,s)$-nets in base $b$ as introduced by Niederreiter \cite{nie87} (see also \cite{DP2010} and \cite[Chapter~4]{niesiam} for surveys of this theory).
\begin{definition}[$(t,m,s)$-nets]\rm
Let $b \ge 2$, $s \ge 1$ and $0 \le t \le m$ be integers. A point set $\cP$ consisting of $b^m$ points in $[0,1)^s$ forms a {\it $(t,m,s)$-net in base $b$}, if every subinterval of the form $\prod_{i=1}^s [a_i b^{-d_i},(a_i+1) b^{-d_i})$ of $[0,1)^s$, with integers $d_i \ge 0$ and integers $0 \le a_i < b^{d_i}$ for $i =1,\ldots, s$ and of volume $b^{t-m}$, contains exactly $b^t$ points of $\cP$.
\end{definition}
In the definition above, $t$ is often called the {\it quality parameter} of the net. It is clear from the definition that any point set consisting of $b^m$ elements from $[0,1)^s$ is at least an $(m,m,s)$-net in base $b$. Smaller values of $t$ imply stronger equidistribution properties for nets.

Infinite versions of $(t,m,s)$-nets are $(t,s)$-sequences as introduced by Niederreiter~\cite{nie87} and their generalizations $(\Tfett,s)$-sequences as introduced by Larcher and Niederreiter~\cite{LN95}. We present here a revised version introduced in \cite{NieXi96} and occasionally called `digital sequences in the broad sense' that eliminates certain problems with ambiguous digit expansions. 
For this we introduce the truncation operator 
\[ [x]_{b,m}:=\sum_{i=1}^{m}x_{i}b^{-i} ,\quad \text{ where }x=\sum_{i\ge 1}x_{i}b^{-i}\in[0,1).\] 
Note that the definition is dependent on the given expansion.
 We also extend this operator to act on vectors by component-wise application.

\begin{definition}[$(\Tfett,s)$-sequences]\rm
Let $b \ge 2$, $s \ge 1$ be integers and let $\Tfett: \NN_0 \rightarrow \NN_0$ be a function which satisfies $\Tfett(m) \le m$ for all $m \in \NN_0$. A sequence $\cS=(\bsx_0,\bsx_1,\ldots)$ of points in $[0,1)^s$ forms a {\it $(\Tfett,s)$-sequence in base $b$}, if for all $m,k \in \NN_0$, the point set consisting of the points $[\bsx_{k b^m}]_{b,m},\ldots,[\bsx_{(k+1) b^m -1}]_{b,m}$ forms a $(\Tfett(m),m,s)$-net in base $b$.  
\end{definition}
In the definition above, $\Tfett$ is often called the {\it quality function} of the sequence. Any sequence in $[0,1)^s$ is at least a $(\Tfett,s)$-sequence in base $b$ with $\Tfett(m)=m$ for all $m \in \NN_0$. Smaller values of $\Tfett$ imply stronger equidistribution properties for sequences. A $(\Tfett,s)$-sequence in base $b$ is called a {\it strict} $(\Tfett,s)$-sequence in base $b$ if 
the value of $\Tfett(m)$ can not be decreased at least by $1$ for any $m\in\NN_{0}$.
It is known from \cite[Theorem~4.32]{DP2010} that a strict $(\Tfett,s)$-sequence in base $b$ is uniformly distributed modulo one if $$\lim_{m \rightarrow \infty} m-\Tfett(m) =\infty.$$ If there exists a $t \in \NN_0$ such that $\Tfett(m)=t$ for all $m \in \NN_0$ then one speaks also of  $(t,s)$-sequences. Obviously, $(t,s)$-sequences are uniformly distributed modulo one.

Explicit constructions of $(t,m,s)$-nets and $(\Tfett,s)$-sequences are usually based on a digital method over a finite field which will be explained in the following:  

From now on let $q$ be a prime-power and let $\FF_q$ be the finite field of order $q$. Let $\varphi: \{0,1,\ldots,q-1\} \rightarrow \FF_q$ be a fixed bijection with $\varphi(0)=0$.  

\begin{definition}[digital $(t,m,s)$-nets over $\FF_q$]\label{defdignet} \rm Let $s,m \in \NN$. Let $C_1,\ldots ,C_s$ be $m \times m$ matrices over $\FF_q$. For $i=1,\ldots,s$ and for $k \in \{0,1,\ldots,q^m-1\}$ with $q$-adic expansion $k=\kappa_0+\kappa_1 q+\cdots +\kappa_{m-1} q^{m-1}$ with $\kappa_0,\ldots,\kappa_{m-1} \in \{0,1,\ldots,q-1\}$  multiply
the matrix $C_i$ by the vector $\vec{k}=(\varphi(\kappa_0),\ldots,\varphi(\kappa_{m-1}))^{\top}$, i.e., $$C_i  \vec{k}=:(y_{i,1}(k),\ldots
,y_{i,m}(k))^{\top} \in (\FF_q^m)^{\top},$$ and set
$$x_{k,i}:=\frac{\varphi^{-1}(y_{i,1}(k))}{q}+\cdots
+\frac{\varphi^{-1}(y_{i,m}(k))}{q^m}.$$ If for some integer $t$ with $0 \le t
\le m$ the point set consisting of the points $$\bsx_k =
(x_{k,1}, \ldots ,x_{k,s})\;\;\mbox{ for }\;\; k \in  \{0,1,\ldots,q^m-1\}$$ is a
$(t,m,s)$-net in base $q$, then it is called a {\it digital $(t,m,s)$-net over
$\FF_q$}, or, in brief, a {\it digital net}. The matrices $C_1,\ldots,C_s$ are called its
\emph{generating matrices} and the matrix $C=(C_1^{\top}C_2^{\top} \ldots C_s^{\top}) \in \FF_q^{m \times s m}$ is called the {\it overall generating matrix} of the digital net.  
\end{definition}
 
The construction of infinite sequences follows more or less the same lines as above, more detailed:

\begin{definition}[digital $(\Tfett,s)$-sequences over $\FF_q$]\label{defdignet} \rm Let $s\in \NN$. Let $C_1,\ldots ,C_s$ be $\NN \times \NN$ matrices over $\FF_q$. For $i=1,\ldots, s$ and for $k \in \NN_0$ with $q$-adic expansion $k=\kappa_0+\kappa_1 q+\cdots$ with $\kappa_0,\kappa_1,\ldots \in \{0,1,\ldots,q-1\}$ multiply
the matrix $C_i$ by the vector $\vec{k}=(\varphi(\kappa_0),\varphi(\kappa_1),\ldots)^{\top}$, i.e., $$C_i  \vec{k}=:(y_{i,1}(k),y_{i,2}(k),\ldots)^{\top} \in (\FF_q^{\NN})^{\top},$$ and set
$$x_{k,i}:=\frac{\varphi^{-1}(y_{i,1}(k))}{q}+\frac{\varphi^{-1}(y_{i,2}(k))}{q^2}+\cdots.$$ If for some function $\Tfett:\NN_0 \rightarrow \NN_0$ with $\Tfett(m) \le m$ for all $m \in \NN_0$ the sequence consisting of the points $$\bsx_k =
(x_{k,1}, \ldots ,x_{k,s})\;\;\mbox{ for }\;\; k \in  \NN_0$$ is a
$(\Tfett,s)$-sequence in base $q$, then it is called a {\it digital $(\Tfett,s)$-sequence over
$\FF_q$}, or, in brief, a {\it digital sequence}. The matrices $C_1,\ldots,C_s$ are called its \emph{generating matrices}.  
\end{definition}

It is known from \cite[Theorem~4.86]{DP2010} that a strict digital $(\Tfett,s)$-sequence in base $b$ is uniformly distributed modulo one, if and only if $$\lim_{m \rightarrow \infty} m-\Tfett(m) =\infty.$$

Many constructions of digital nets are inspired by a close connection between coding theory and the theory of digital nets (see, for example, Niederreiter \cite{Ni2004,N05,N06a}). Examples for that are the so-called $(u,u+v)$-construction (see \cite{BES02,np}), the matrix-product construction (see \cite{NieOe04}) and the Kronecker-product construction (see \cite{BES02,NiePir02}). Here we deal with a construction for digital nets which is an analog to a special type of codes, namely to cyclic codes which are well known in coding theory. This construction has been introduced by Niederreiter in \cite{Ni2004} who adopted the viewpoint that cyclic codes can be defined by prescribing roots of polynomials. Later this construction has been generalized by Pirsic, Dick and Pillichshammer \cite{PDP} to so-called hyperplane nets whose definition will be given in Section~\ref{secHNet} where also some results on the quality parameter are recalled.

It is the aim of this paper to introduce a corresponding construction also for digital sequences and to prove some results concerning the equidistribution of the resulting sequences. Such a construction can be obtaind with the help of duality theory for digital sequences as recently introduced by Dick and Niederreiter~\cite{DN2009}. 

The paper is organized as follows: In Section~\ref{sec:dlty} we recall the basic facts from duality theory for digital nets and sequences. This theory will be an important tool in our construction. Furthermore, in Section~\ref{secHNet} we recall the definition of and some basic results on hyperplane nets. The construction of what we call hyperplane sequences will be presented in Section~\ref{defhypseq}. In this section we also show how to determine the quality function of hyperplane sequences and we prove an `if and only if' condition under which hyperplane sequences are uniformly distributed modulo one. In Section~\ref{secEGM} we determine the generator matrices of hyperplane sequences and in Section~\ref{secRelLN} we compare hyperplane sequences to Larcher-Niederreiter sequences. Last but not least, in Section~\ref{secExRes}, we show the existence of hyperplane sequences which satisfy $\Tfett(m) =s \log_q m +2\log_q\log m+O(1)$. This in turn implies the existence of hyperplane sequences which satisfy a certain bound on the star discrepancy.

\section{Duality Theory}\label{sec:dlty}

Duality theory is an important tool for the construction of digital nets and sequences and provides a connection to the theory of linear codes. It has been introduced by Niederreiter and Pirsic~\cite{NiePir02} for digital nets and by Dick and Niederreiter~\cite{DN2009} for digital sequences. We also refer to \cite[Chapter~7]{DP2010} for an overview. In the following we recall the basic facts which will be necessary for the construction of hyperplane sequences in Section~\ref{defhypseq}.

\subsection{Duality Theory for Digital Nets}

Let $\cN$ be a $\FF_q$-linear subspace of $\FF_q^{sm}$. Then its dual space $\cN^{\bot} \subseteq \FF_q^{sm}$ is its orthogonal complement relative to the standard inner product in $\FF_q^{sm}$, i.e., $$\cN^{\bot}=\{\bfA \in \FF_q^{sm}\ : \ \bfB \cdot \bfA =0 \mbox{ for all } \bfB \in \cN\}.$$

\begin{definition}[NRT weight]\rm \label{defNRTweight}
For $\bfa=(a_1,\ldots ,a_m)\in \FF_q^m$ let $$v_m(\bfa)=\left\{
\begin{array}{ll}
0 & \mbox{ if } \bfa =\bszero,\\
\max\{j \, : \, a_j \not=0\} & \mbox{ if } \bfa \not=\bszero.
\end{array}\right.$$ 
We extend this definition to $\FF_q^{sm}$ by writing $\bfA \in \FF_q^{sm}$ as the concatenation of $s$ vectors of length $m$, i.e., $\bfA=(\bfa_1,\ldots,\bfa_s) \in \FF_q^{sm}$ with $\bfa_i\in \FF_q^m$ for $i=1,\ldots, s$, and putting $$V_m(\bfA)=\sum_{i=1}^s v_m(\bfa_i).$$ The weight $V_m$ is called the {\it Niederreiter-Rosenbloom-Tsfasman weight}. 
\end{definition}

\begin{definition}[minimum distance]\label{defNRTmindist}\rm
Let $\cN \not=\{\bszero\}$ be a $\FF_q$-linear subspace of $\FF_q^{sm}$. Then the {\it minimum distance} of $\cN$ is defined by $$\delta_m(\cN)=\min\{V_m(\bfA) \, : \, \bfA \in \cN \setminus \{\bszero\}\}.$$ Furthermore, we put $\delta_m(\{\bszero\})=sm+1$ by convention.
\end{definition}

The following theorem has been first shown by Niederreiter and Pirsic~\cite{NiePir02}. It can be also found as \cite[Corollary~7.12]{DP2010}.
\begin{theorem}
Let $m,s \in \NN$, $s \ge 2$. Then, from any $\FF_q$-linear subspace $\cN$ of $\FF_q^{sm}$ with $\dim(\cN) \ge sm-m$ one can construct generating matrices $C_1,\ldots,C_s$ of a strict digital $(t,m,s)$-net over $\FF_q$ such that $$t=m-\delta_m(\cN)+1.$$ The row space $\Ccal$ of the overall generating matrix $C$ satisfies $\Ccal=\cN^{\bot}$.
\end{theorem}

\subsection{Duality Theory for Digital Sequences}

A digital sequence over $\FF_q$ is fully determined by its generating matrices $C_1,\ldots, C_s \in \FF_q^{\NN \times \NN}$. For $m \in \NN$ we denote the $m \times m$ left-upper sub-matrix of $C_i$ by $C_i^{(m)}$. The matrices $C_1^{(m)},\ldots, C_s^{(m)}$ are then the generating matrices of a digital $(t,m,s)$-net over $\FF_q$. The overall generating matrix of this digital net is defined by
$$C^{(m)}=(C_1^{(m)\top} C_2^{(m)\top}\ldots C_s^{(m)\top})\in \FF_q^{m \times sm}$$ for any $m \in \NN$.

Hence a digital sequence can equivalently be described via the sequence $C^{(1)},
C^{(2)}, \ldots$ of overall generating matrices or the sequence $\Ccal_1,\Ccal_2,
\ldots$ of row spaces thereof. Each 
$C^{(m)}$ or $\Ccal_m \subseteq \FF_q^{sm}$
describes the first $q^m$ points of the digital sequence and each has a dual space 
$\Ccal^\perp_m \subseteq \FF_q^{sm}$ (which is
the null space of $C^{(m)}$ in $\FF_q^{sm}$) 
associated with it (see \cite[Section~7.1]{DP2010}). Hence the 
dual for a digital sequence consists of
a sequence $(\Ccal^\perp_m)_{m \ge 1}$ of dual spaces which have
certain relations to each other, that we shall explain in the following, crucial definition.

\begin{definition}[dual space chain] \label{defdualseq}
\rm 
Let $s \in \NN$, $s \ge 2$. For all $m \in \NN$, let $\cN_m$ be a $\FF_q$-linear subspace of
$\FF_q^{sm}$ with $\mathrm{codim}(\cN_m) \le m$. Let 
$$\cN_{m+1,m}:=\left\{(\bfa_1,\ldots, \bfa_s)\in\FF_{q}^{s m}:
((\bfa_1, 0),\ldots, (\bfa_s, 0))\in\cN_{m+1}
\right\},$$
i.e., $\cN_{m+1,m}$ is the projection of vectors in $\cN_{m+1}$ with zeroes at every $(m+1)$-th coordinate.
 Suppose that $\cN_{m+1,m}$ is a $\FF_q$-linear subspace of $\cN_m$ with 
$\dim (\cN_{m}/\cN_{m+1,m}) \le 1$ for all $m \in \NN$.
Then the sequence $(\cN_m)_{m \ge 1}$ of spaces is called a 
\emph{dual space chain}.
\end{definition}

The following result has been first shown by Dick and Niederreiter~\cite{DN2009}. It can also be found as \cite[Corollary~7.27]{DP2010}.

\begin{theorem}\label{chp5th2}
For a given dual space chain $(\cN_m)_{m \ge 1}$, one can construct generating matrices
$C_1,\ldots,C_s$ of a strict digital $(\Tfett,s)$-sequence over $\FF_q$ such that $$\Tfett(m)=m-\delta_m(\cN_m)+1\;\mbox{ for all }\; m \in \NN.$$ For all $m \in \NN$ the $m$th row space $\Ccal_m$ of the $m$th overall generating matrix $C^{(m)}$  satisfies $\Ccal_m =\cN_m^\perp$. 
\end{theorem}

\section{Hyperplane Nets}\label{secHNet}

Hyperplane nets are an example of digital nets whose construction is based on duality theory. They are a generalization of cyclic nets as introduced by Niederreiter~\cite{Ni2004} and were first presented in \cite{PDP}. The slightly more general definition given here was first presented in \cite{cyczoo}. An overview of hyperplane nets can be found in \cite[Chapter~11]{DP2010}.

\begin{definition}[hyperplane nets]\label{defhypnet}\rm
Let $m,s \in \NN$, $s\geq2$, and a prime-power $q$ be given. Let $f \in \FF_q[x]$ with $\deg(f)=m$ and let $R_m:=\FF_{q}[x]/(f)$ be the polynomial residue class ring modulo the ideal $(f)$. For $i=1,\ldots , s$ let $\mathfrak B_i=\{\mathfrak b_{i,1},\ldots,\mathfrak b_{i,m}\}$ be an ordered basis of the polynomial residue class ring $R_m$ considered as vector space over $\FF_q$.

Fix an element $\bsalpha =(\alpha_1,\ldots,\alpha_s) \in R_m^s \setminus \{\bszero\}$ and consider the subspace $$\mathcal{N}'_{\bsalpha}:=\{ \bfk \in R_m^s\, : \, \bsalpha \cdot \bfk \equiv 0 \pmod{f}\}.$$
Define the mapping $\theta: R_m^s \rightarrow \FF_q^{sm}$ by
\[
\bfk=(k_1,\ldots,k_s) \in R_m^s \mapsto
   (\kappa_{1,1},\ldots,\kappa_{1,m},\ldots,\kappa_{s,1},\ldots,\kappa_{s,m})
  \in\FF_q^{sm}, \]
where $(\kappa_{i,1},\ldots,\kappa_{i,m})$ is the coordinate
vector of $k_i \in R_m$ with respect to the basis $\mathfrak
B_i$.

We denote by $\mathcal C_{\bsalpha}$ the orthogonal subspace in
$\FF_q^{sm}$ of the image $\mathcal N_{\bsalpha}:=\theta(\mathcal
N'_{\bsalpha})$. Let \[C_{\bsalpha} \in
\FF_q^{m\times sm}\] be a matrix whose row space is $\mathcal
C_{\bsalpha}$. Then we call the digital net with overall generating matrix $C_{\bsalpha}$ a
\emph{hyperplane net over $\FF_q$ with respect to ${\mathfrak B_1,\ldots,\mathfrak B_s}$}. This hyperplane net will be denoted by $\cP_{\bsalpha}$ and we say $\cP_{\bsalpha}$ is the hyperplane net associated with $\bsalpha$. 
\end{definition}

Note that the dual space of a hyperplane net is 
\begin{eqnarray*}
\mathcal{N}_{\bsalpha} = \theta(\mathcal
N'_{\bsalpha})= \{\bfA \in \FF_{q}^{sm} \, : \, C_{\bsalpha} \bfA^{}=\bszero\}.
\end{eqnarray*}

In the following we show how to determine the quality-parameter $t$ of a hyperplane net. To this end we need a further definition:

\begin{definition}[figure of merit]\rm
For $\bsalpha \in R_m^s$ the {\it figure of merit} is defined as 
\begin{equation}\label{figmer}
\rho(\bsalpha)=s-1+\min_{\bfk \in \mathcal N'_{\bsalpha}\setminus \{\bszero\}} \sum_{i=1}^s \deg(k_i) = -1 +\delta_m(\cN_{\bsalpha}).
\end{equation}
\end{definition}

The following result gives a relation between the quality-parameter of a hyperplane net and its figure of merit. It can be found as \cite[Theorem~1]{PP2009}.

\begin{theorem}
A hyperplane net $\cP_{\bsalpha}$ associated to $\bsalpha \in R_m^s\setminus\{\bszero\}$ is a digital $(t,m,s)$-net over $\FF_q$ with $t=m-\rho(\bsalpha)$.
\end{theorem}

Observe that there is no dependence on the chosen bases. This means that a different choice of bases can be considered as a scrambling of the original sequence that does not deteriorate the
quality.

The following existence result is a slight generalization of \cite[Theorem~2]{PP2009}.

\begin{theorem}\label{llnsth1hypnet}
Let $m,s \in \NN$, $s \ge 2 $ and let $q$ be a prime-power. Choose ordered bases $\mathcal{B}_1,\ldots,\mathcal{B}_s$ of $R_m$ over $\FF_q$. For $\rho \in \ZZ$ define $$\Delta_q (s,\rho) =\sum_{d=0}^{s-1}{s \choose d}(q-1)^{s-d}\sum_{\gamma=0}^{\rho+d}{s-d+\gamma-1 \choose \gamma} q^{\gamma}+1-q^{\rho +s}.
$$
Let $\beta \in \RR$ such that $0 < \beta \le 1$. If $\Delta_q (s,\rho) < \beta q^m \left(\frac{q-1}{q}\right)^{s-1}$, then there exist more than $(1-\beta) ((q-1)q^{m-1})^{s-1}$ elements $\bsalpha = (1,\alpha_2,\ldots,\alpha_s) \in R_m^s$ with $\gcd(\alpha_i,x)=1$ for all $i=2, \ldots, s$, such that $\rho(\bsalpha) \ge s + \rho$. The associated hyperplane nets $\cP_{\bsalpha}$ are digital $(t,m,s)$-nets over $\FF_q$ with $t \leq m-s-\rho.$
\end{theorem}

\begin{corollary}\label{llnsc1}
Let $m,s \in \NN$, $s \ge 2 $ and $m$ sufficiently large, and let $q$ be a prime-power. Let $\beta \in \RR$ such that $0 < \beta \le 1$. There exist more than $(1-\beta) ((q-1)q^{m-1})^{s-1}$ elements $\bsalpha = (1,\alpha_2,\ldots,\alpha_s) \in R_m^s$ with $\gcd(\alpha_i,x)=1$ for all $i=2,\ldots ,s$, such that $$\rho(\bsalpha) \ge \left\lfloor m +\log_q \beta- (s- 1) (\log_q m -1)+ \log _q \frac{(s-1)!}{q^{s-1}}\right\rfloor.$$
\end{corollary}

\begin{proof}
For $\rho \ge 1 $ we have
\begin{eqnarray*}
\Delta_q(s,\rho) & \le & \sum_{d=0}^{s-1}
  {s \choose d} (q-1)^{s-d} {\rho + s -1\choose s - d-1}
   \frac{q^{\rho + d+1}}{q-1}\\
 & \le & q^{\rho+1} \sum_{d=0}^{s-1}{s \choose d} (q-1)^{s-d-1} q^d \frac{(\rho +s-1)^{s-d-1}}{(s-d-1)!}\\
& = & \frac{\rho^{s-1}}{(s-1)!} q^{\rho+1} (q-1)^{s-1} \left(1 + O_{s}\left(\frac{1}{\rho}\right)\right),
\end{eqnarray*}
where $O_{s}$ indicates that the implied constant depends only on $s$.
Now let
$$ \rho = \left\lfloor m+ \log_q \beta - (s- 1) \log_q m + \log _q \frac{(s-1)!}{q^{s-1}} -1\right\rfloor ,$$
which is in $\NN$ for sufficiently large $m$. Then
\begin{eqnarray*}
\Delta_q(s,\rho)
& \le & \beta q^m  \left(\frac{q-1}{q}\right)^{s-1} \left(1+\frac{\log_q \beta}{m} - (s-1) \frac{\log_q m}{m} + \frac{1}{m}
\log_q \frac{ (s-1)!}{q^{s-1}}\right)^{s-1}\\
&& \times \left(1 + O_s\left(\frac{1}{m}\right)\right)\\
& < & \beta q^m \left(\frac{q-1}{q}\right)^{s-1}
\end{eqnarray*}
for sufficiently large $m$. The result follows from Theorem~\ref{llnsth1hypnet}.
\end{proof}

\section{Definition of Hyperplane Sequences}\label{defhypseq}

For $m \in \NN$ set $f_m(x)=x^m \in \FF_q[x]$. For each $i=1,\dots,s$ and all $m\in \NN$ choose a basis $\mathfrak B_{m,i}=\{\mathfrak b_{m,i,1},\dots,\mathfrak b_{m,i,m}\}$ 
of the polynomial residue class ring $R_m:=\FF_q[x]/{(x^m)}$ over $\FF_q$
such that, expressed in the canonical basis, $\{1,x,x^2,\dots\}$, we have
 $\mathfrak b_{m+1,i,j}=(\mathfrak b_{m,i,j},0)$ for all $1 \le j \le m$.
Define the mapping $\theta_m: R_m^s \rightarrow \FF_q^{sm}$ by
\[
\bfk=(k_1,\ldots,k_s) \in R_m^s \mapsto
   (\kappa_{1,1},\ldots,\kappa_{1,m},\ldots,\kappa_{s,1},\ldots,\kappa_{s,m})
  \in\FF_q^{sm}, \]
where $(\kappa_{i,1},\ldots,\kappa_{i,m})$ is the coordinate
vector of $k_i \in R_m$ with respect to the chosen basis $\mathfrak
B_{m,i}$ for $i=1,\ldots, s$.

Let $\FF_q[[x]]$ be the space of formal power series over $\FF_q$. Let $Y$ denote the subset of power series whose reductions modulo any
 $x^m$ are invertible modulo $x^m$, i.e., 
 $$Y=\{\alpha \in \FF_q[[x]]\,:\,  \gcd(\alpha \bmod x^m,x^m)=1 \,\forall m\in \NN\}.$$ 

Let $\bsalpha \in Y^s$ and, for $m \in \NN$, set $\bsalpha^{(m)} \equiv \bsalpha \pmod{x^m}$ (coordinate-wise).

\begin{theorem}\label{hyperplanspchain}
For $\bsalpha \in Y^s$ and $m \in \NN$ let $\mathcal N'_{\bsalpha^{(m)}}=\{\bfk \in R_m^s\, : \, \bsalpha^{(m)} \cdot \bfk \equiv 0 \pmod{x^m}\}$. Then the sequence $(\mathcal N_{\bsalpha^{(m)}})_{m \ge 1}$, where $\mathcal N_{\bsalpha^{(m)}} = \theta_m(\mathcal N'_{\bsalpha^{(m)}})$, is a dual space chain.  
\end{theorem}

\begin{proof}
Let $\mathcal N_{m+1,m}$ be 
the set
$$\cN_{m+1,m}:=\left\{(\bfa_1,\ldots, \bfa_s)\in\FF_{q}^{m s}:
((\bfa_1, 0),\ldots, (\bfa_s, 0))\in\cN_{\bsalpha^{(m+1)}}
\right\}.$$
 By the choice of the sequence of bases $(\mathfrak B_{m,i})_{m \ge 1}$ for $i=1,\ldots, s$, it follows that $\mathcal N_{m+1,m}$ is a $\FF_q$-linear subspace of $\mathcal N_{\bsalpha^{(m)}}$. It remains to show that $$
\dim(\mathcal N_{\bsalpha^{(m)}}/ \mathcal N_{m+1,m}) \le 1 \Leftrightarrow
 \dim(\mathcal N_{m+1,m}) \ge \dim(\mathcal N_{\bsalpha^{(m)}})-1.$$ 
We have $$\mathcal N_{m+1,m} = \theta_m(\{\bfk \in R_m^s \, : \, \bsalpha^{(m+1)}\cdot \bfk \equiv 0 \pmod{x^{m+1}}\}).$$ Now we have $\bfA \in \mathcal N_{m+1,m}$ if and only if $\bfA=\theta_m(\bfk)$ where $\bsalpha^{(m+1)}\cdot \bfk \equiv 0 \pmod{x^{m+1}}$.
Let $\bsalpha^{(m+1)} = (\alpha_1,\ldots,\alpha_s)\in R_{m+1}^s$ where $\alpha_i=\alpha_{i,0}+\alpha_{i,1} x+\cdots +\alpha_{i,m} x^m$ for $i=1,\ldots, s$ and let $\bfk=(k_1\ldots,k_s) \in R_{m}^s$ with $k_i=\kappa_{i,0}+\kappa_{i,1} x +\cdots + \kappa_{i,m-1} x^{m-1}$ for $i=1,\ldots, s$. Then $\bsalpha^{(m+1)}\cdot \bfk \equiv 0 \pmod{x^{m+1}}$ is equivalent to
\begin{eqnarray}\label{cond1}
\sum_{j = 0}^m \left(\sum_{i=1}^s\sum_{l,u \ge 0 \atop l+u=j}\kappa_{i,l} \alpha_{i,u}\right)x^j =0.
\end{eqnarray}
Hence $\mathcal N_{m+1,m}$ is isomorphic to the nullspace of the $(m+1) \times sm$ matrix $A$ over $\FF_q$ given by $$\sum_{i=1}^s\left(\begin{array}{lllll}
         \alpha_{i,0} & 0         & \ldots & \ldots & 0\\
         \alpha_{i,1} & \alpha_{i,0}  & 0      & \ldots & 0\\
         \multicolumn{5}{c}\dotfill\\
         \alpha_{i,m-1} & \alpha_{i,m-2} & \alpha_{i,m-3} & \ldots & \alpha_{i,0}\\ 
         \alpha_{i,m} & \alpha_{i,m-1} & \alpha_{i,m-2} & \ldots & \alpha_{i,1}
        \end{array}\right).$$ Therefore we obtain $$\dim(\mathcal N_{m+1,m}) \ge sm-m-1 = \dim(\mathcal N_{\bsalpha^{(m)}})-1$$ and the result follows. 
\end{proof}

Combining Theorem~\ref{chp5th2}, Theorem~\ref{hyperplanspchain} and \eqref{figmer} we obtain the following corollary.

\begin{corollary}\label{co1}
From any $\bsalpha \in Y^s$ one can construct a strict digital $(\Tfett,s)$-sequence $\mathcal S_{\bsalpha}$ over $\FF_q$ where $$\Tfett(m)=m-\rho(\bsalpha^{(m)}) \ \mbox{ for all }\ m \in \NN.$$ 
\end{corollary}

\begin{definition}[hyperplane sequence]\rm
The strict digital $(\Tfett,s)$-sequence $\mathcal S_{\bsalpha}$ over $\FF_q$ is called a {\it hyperplane sequence} associated with $\bsalpha$.
\end{definition}

\begin{theorem}\label{udt}
Let $\bsalpha=(\alpha_1,\ldots,\alpha_s)\in Y^s$. The hyperplane sequence $\mathcal S_{\bsalpha}$ over $\FF_q$ is uniformly distributed modulo one, if and only if $\alpha_1,\ldots,\alpha_s$ are linearly independent over $\FF_q[x]$.
\end{theorem}

\begin{proof}
By \cite[Theorem~4.86]{DP2010} a strict digital $(\Tfett,s)$-sequence over $\FF_q$ is uniformly distributed modulo one, if and only if $\lim_{m \rightarrow \infty} m-\Tfett(m)=\infty$. Hence, by Corollary~\ref{co1}, the sequence $\mathcal S_{\bsalpha}$ is uniformly distributed modulo one, if and only if $$\lim_{m \rightarrow \infty}\rho(\bsalpha^{(m)})=\infty.$$

Let $\mathcal S_{\bsalpha}$ be uniformly distributed modulo one. Assume that there exist $p_1,\ldots,p_s \in \FF_q[x]$, not all of them the zero polynomial, such that $$p_1\alpha_1+\cdots+p_s \alpha _s =0 \in \FF_q[[x]].$$ Then for $m > \max_{1 \le i \le s} \deg(p_i)$ we have $(p_1,\ldots,p_s) \in \mathcal N'_{\bsalpha^{(m)}}\setminus  \{\bszero\}$ and hence $$\rho(\bsalpha^{(m)}) \le s-1+\sum_{i=1}^s \deg(p_i).$$ Thus, $\rho(\bsalpha^{(m)})$ is bounded as $m$ tends to infinity which contradicts the uniform distribution of $\mathcal S_{\bsalpha}$.

Let $\bsalpha=(\alpha_1,\ldots,\alpha_s) \in Y^s$ be such that $\alpha_1,\ldots,\alpha_s$ are linearly independent over $\FF_q[x]$. Assume that the sequence $\mathcal S_{\bsalpha}$ is not uniformly distributed modulo one. Then there exists a $K \in \NN$ such that $$\rho(\bsalpha^{(m)}) < K \ \ \ \mbox{ for all }\ \ \ m \in \NN.$$ From the definition of the figure of merit  $\rho$ it is evident that a bound on $\rho$ implies a bound on the degrees of the possible $p_i$ for a nontrivial 
linear dependence relation. Thus it follows by the pigeonhole principle that there exist $p_1,\ldots,p_s \in \FF_q[x]$, not all of them the zero polynomial, such that $$p_1 \alpha_1^{(m)}+\cdots +p_s \alpha_s^{(m)} \equiv 0 \pmod{x^m}$$ for infinitely many $m \in \NN$. This however leads to a contradiction, since by the linear independence of $\alpha_1,\ldots,\alpha_s$ over $\FF_q[x]$ it follows that $p_1\alpha_1+\cdots +p_s \alpha_s \not=0 \in \FF_q[[x]]$. Hence there exist $l \in \NN_0$ and $a_i \in \FF_q$ for $i \ge l$ and $a_l \not=0$, such that $$ p_1\alpha_1+\cdots +p_s \alpha_s =a_l x^l+a_{l+1} x^{l+1}+\cdots$$ and therefore $$p_1\alpha_1^{(m)}+\cdots +p_s \alpha_s^{(m)} \not \equiv 0 \pmod{x^m} \ \ \ \mbox{ for all }\ \ \ m >l.$$
\end{proof}

\section{Explicit Generator Matrices}\label{secEGM}

First we briefly recall the explicit form of hyperplane net generator matrices
 (cf. \cite{PDP} and \cite[Section~11]{DP2010}).
From this, the shape of the  infinite generator matrices of hyperplane sequences will then be evident.

 Going back to the general ring $R_m=\FF_q[x]/(f)$ of Definition \ref{defhypnet}, 
 an injective ring homomorphism $\Psi:R_m\to\FF_q^{m\times m}$ can be constructed by setting 
  $\Psi(x)$ equal to the companion matrix of the generating polynomial  $f(x)$ and, 
 for all further elements, requesting the ring homomorphism property,
  $\forall x,y,z\in R_m:\: 
  \Psi(x y-z)=\Psi(x)\Psi(y)-\Psi(z)$. The validity of this construction can be realized using the fact that the minimal polynomial of a companion matrix is the polynomial itself. Also, it is easily verified that this homomorphism condition is equivalent to setting
 \[\forall a_1,\dots,a_m\in\FF_q:\: \Psi\left(\sum_{i=0}^{m-1}a_ix^i\right)=\sum_{i=0}^{m-1}a_i\Psi(x)^i.\] 
 Next, take the canonical basis 
 $\mathfrak E_i=\{1,x,\dots,x^{m-1}\}$ for $R_m$,  as a $\FF_q$-space. For $i=1,\dots,s$ let $B_i$ be the matrix associated to the basis change from the given basis 
 $\mathfrak B_i$ to $\mathfrak E_i$. Then the following was proven in \cite{PDP} (the version given here follows \cite[Theorem~11.5]{DP2010}; the theorem as originally stated in \cite{PDP} had a minor flaw):
 \begin{theorem}
  The  matrices $C_i:=(\Psi(\alpha_i)B_i)^\top,\,i=1,\dots,s$ can be chosen as generating
  matrices for the hyperplane net $\cP_{\bsalpha}$.
\end{theorem}

Now note that in the case of hyperplane sequences the matrices $B_{m,i}$ 
associated to the bases $\mathfrak B_{m,i}$ are upper triangular: this follows inductively, since 
the columns of $B_{m,i}$ for $m\geq 0$ and $i=1,\dots,s$ consist of the coordinates of $\mathfrak b_{m,i,j},\, j=1,\dots,m$ with respect to the canonical basis and we requested the last component to be equal to $0$ unless $j=m$. 
Let $\pi_m$ denote the action of taking the upper left $m\times m$ submatrix of a finite or infinite matrix.  
Then it also follows inductively that for $m'\leq m$ we have $\pi_{m'}( B_{m,i} ) = B_{m',i}$.
Hence there are infinite upper triangular matrices $B_{i}\in\FF_q^{\NN\times\NN}$ such that $\pi_m(B_i)=B_{m,i}$.

Denote with $\Psi^{(m)}$ the ring homomorphism as discussed above, associated to the 
extension $\FF_{q}[x]/(x^{m})$
 used in the construction of the hyperplane net  $\cP_{\bsalpha^{(m)}}$. It can be given explicitly
 quite easily:  first note that the companion matrix of the polynomial $x^{m}$ is the matrix
with entries of $1$ in the first lower subdiagonal and $0$ elsewhere,
so $\Psi^{(m)}(x)=(\delta_{i,j+1})_{i,j=0,\dots,m-1}$. It follows that $\Psi^{(m)}(x^{k})=
(\delta_{i,j+k})_{i,j=0,\dots,m-1}$ for $k=0,\dots,m-1$ and consequently
\[
  \Psi^{(m)}(\alpha_i^{(m)} )= \begin{pmatrix}
    a_{i,0}  & 0        & 0    & \dots & 0\\
    a_{i,1} & a_{i,0} & 0   & \dots  & 0\\
     \vdots &              &     &            & \vdots \\
     a_{i,m-1} & a_{i,m-2} & \dots & \dots & a_{i,0}
  \end{pmatrix}\ \
 \text{ where }\ \ 
 \alpha_i^{(m)}=\sum_{j\geq0} a_{i,j}x^j,
\]
so $\Psi^{(m)}(\alpha_i^{(m)})$ is a lower triangular Toeplitz matrix. We see that each
of these matrices is the upper left $m\times m$ submatrix of an (also lower triangular Toeplitz) infinite matrix in $\FF_q^{\NN\times\NN}$ which we denote as $\Psi(\alpha_i)$.

Thus we have shown that the generator matrices of $\cP_{\bsalpha^{(m)}}$ can be chosen 
as $C_{m,i}:=(\pi_m(\Psi(\alpha_i)) \pi_m(B_i))^\top$. However, by the triangular structure of the
matrices also the product $\Psi(\alpha)B_i$ is admissible and we have furthermore
$C_{m,i}= \pi_m(\Psi(\alpha_i) B_i)^\top$. This implies that the first block
$[\bsx_{0}]_{b,m},\dots,[\bsx_{b^{m}-1}]_{b,m}$ of the sequence $\cP_{\bsalpha}$ coincides with
$\cP_{\bsalpha^{(m)}}$. We summarize:

\begin{theorem}
Let $\bsalpha \in Y^s$ and let $\mathfrak B_{m,i}$ be as in Section \ref{defhypseq}. For $i=1,\ldots,s$ define the matrices $B_i\in\FF_q^{\NN\times\NN}$ by setting, for $j\in \NN$, the first $j$ entries of the $j$th column equal to the coefficients of $\mathfrak b_{j,i,j}$ modulo $x^j$ and zero elsewhere. For $\alpha_i=\sum_{j\geq0}a_{i,j}x^{j}\in\FF_q[[x]]$ for $i=1,\dots,s$ associate infinite lower triangular Toeplitz matrices $\Psi(\alpha_i)\in\FF_q^{\NN\times\NN}$ by 
\[
  \Psi^{}(\alpha_i^{} ):= \begin{pmatrix}
    a_{i,0}  & 0        & 0    & \dots  & \dots \\
    a_{i,1} & a_{i,0} & 0   & \dots  &  \\
     \vdots &      \ddots        &  \ddots   & \ddots&  \vdots  \\
     a_{i,m-1} & a_{i,m-2} & \dots  & a_{i,0} &  \ddots\\
     \vdots & & & & \ddots 
  \end{pmatrix}.
\]
Then,  the generator matrices of $\cP_{\bsalpha}$ are given 
by $C_{i}=(\Psi(\alpha_i^{})B_i)^\top$.
\end{theorem}

\begin{remark}\rm
  In particular cases, e.g., if the $B_{i}$ have finite columns (or, equivalently, all $\mathfrak b_{i,j}$
   are polynomials) all generator matrices have the finite-column property, i.e., almost all entries
  are $0$. This is the property required for the definition of `digital sequences in the narrow sense'.
  Hence, if the bases include nonpolynomial Laurent series, and consequently
   the matrices would lack this requirement of finite columns, then for sequences in the `narrow' 
   sense we would only be able to state:
  for $m> 0$, the generator matrices of $\pi_m(\cP_{\bsalpha^{(m)}})$ can be chosen
 as $C_{m,i}=\pi_m(\Psi(\alpha_i^{})B_i)^\top$.
\end{remark}\begin{remark}\rm
 Note that $C_{i}=B^\top \Psi(\alpha_i)^\top$ is the product of a lower triangular matrix with 
 an upper triangular Toeplitz matrix and can therefore be interpreted as a linear scrambling 
 (in the sense of Faure and Tezuka \cite{FauTez}) of the  matrices $\Psi(\alpha_i)^\top$.
\end{remark}

\section{Relation to Larcher-Niederreiter Sequences}\label{secRelLN}

 Recall that Larcher-Niederreiter sequences, as defined in \cite[Ch.~4.4, Eq.~4.64]{niesiam} 
 and treated  in detail in \cite{LarNie} are based on elements 
 chosen from $\FF_q((z))$ (where $z:=1/x$), or rather, actually the subring  
 $z\FF_{q}[[z]]$ is used for the definition. 
 It is a natural question to ask what the relation to  hyperplane sequences is.
 
   The shortest way to describe Larcher-Niederreiter sequences
  is by defining them as being generated by arbitrary infinite Hankel matrices; in the 
  construction cited above, the entries of a matrix map bijectively to the positive-degree
  coefficients of a Laurent series.
    It is well-known that finite Hankel matrices correspond to polynomial lattice point sets
   (PLPSs), see \cite[Chapter~10]{DP2010}. Hence initial $b$-adic blocks of Larcher-Niederreiter sequences are, up to truncation, equivalent to PLPSs.
 
  It has been shown in \cite[Theorem~2]{cyczoo} that PLPSs are also a special case of
  hyperplane nets, when choosing the canonical basis: This choice results in the
  $B_i$ being identity matrices. So in the case of the $\cP_{\bsalpha^{(m)}}$ (with the canonical  
  basis) for any $m>0$ we have $C_{m,i}=\Psi(\alpha_i^{(m)})^\top$
   as generator matrices. Let $P^{(m)}$ be the horizontally reflected
   identity matrix. Then the matrices $C_{m,i} P^{(m)}$ still
   generate the same net, since the regular matrix $P^{(m)}$ only effects a reordering of the point set,   
   but the matrices are now (triangular) Hankel matrices, hence generate a PLPS.
   
   We will show in the following that, however, the matrices $C_{i}=$ ``$\lim_{m }C_{m,i}$''
    do \emph{not} generate a Larcher-Niederreiter sequence
   (to motivate this intuitively, consider that, as $m$ grows, the matrices $P^{(m)}$ shift the first few points of   $\cP_{\bsalpha^{(m)}}$   ever further to the end of the sequence, effectively removing them in the limit). 
     For our proof we first restate \cite[Theorem~3]{DN2009}
     (cf. also \cite[Th.~7.28]{DP2010}) and add a corollary.
   
   \begin{theorem}\label{thm9} 
    Let $(\cN_m)_{m\geq 1}$ be a dual space chain. The generating matrices $C_1,...,C_s$ are
    unique up to a multiplication of $C_1,...,C_s$ from the right with the same nonsingular upper triangular (NUT) matrix, if and only if 
   $\dim(\cN_{m+1,m}) = sm - m - 1$ for all $m \in \NN$
    \end{theorem}
   
   \begin{corollary}
   If a sequence is generated by regular matrices $C_{1},\dots,C_{s}$ that fulfill 
   \begin{equation}\label{stern}\mathrm{rank}( C^{(m,m+1)})=\mathrm{rank}( C^{(m)})+1=m+1 
   \ \ \mbox{ for all }\ \ m\in \NN,\end{equation}  where
   $C^{(m,m+1)} := (C_{1}^{(m,m+1)\top},\dots,C_{s}^{(m,m+1)\top})$
   and the $C_{i}^{(m,m+1)}$ are the left upper $m\times (m+1)$ submatrices of $C_{i}$,
   then any $s$ matrices generating the same sequence have the form $C'_{i}=C_{i}P$,
   with $P$ an arbitrary NUT matrix.
   \end{corollary}
   
   \begin{proof}
   This follows from Theorem~\ref{thm9} by noting, as in Section \ref{sec:dlty}, 
   that  $\cN_{m}$ is the dual space of  $\mathcal C_{m},$ the row space of $C^{(m,m)}$. 
   Due to the regularity $\mathcal C_{m}$ has dimension $m$, hence its dual $\cN_{m}$ has
   dimension $sm-m$. Furthermore, $\cN_{m+1,m}$ is the dual space of the
   row space of $C^{(m,m+1)}$. Hence, if the rank condition is fulfilled, the dimension of
   $\cN_{m+1,m}$ is $sm-\mathrm{rank}(C^{(m,m+1)})=sm-m-1$ and the theorem can be applied.
   \end{proof}
   
   Now what is left is to verify the condition \eqref{stern} for the $C_{m,i}$ as given above. Then, by
   the corollary, any matrices generating the same sequence would have to be upper triangular,
   hence are never Hankel matrices as would be the case for Larcher-Niederreiter sequences.
   
  Note that the first $m$ row vectors of $C_{i}^{(m,m+1)}$ are the $m$-dimensional coefficient
  vectors of $x^{j}\alpha_{i}^{(m)} \pmod{x^m}$ and the final, 
   $(m+1)$-th row vector is equal to $(\alpha_{i}^{(m+1)}-a_{i,0})/x$. Supposing a nontrivial linear 
    combination of the row vectors of $C^{(m,m+1)}$ exists, then, for some
     $\lambda_{j}\in\FF_{q},\,j=0,\dots,m-1$ and for each $i=1,\dots,s$ we get
  \begin{align*}
    & \sum_{j=0}^{m-1} \lambda_{j} x^{j} \alpha_{i}^{(m)} = (\alpha_{i}^{(m+1)}-a_{i,0})/x &\pmod{x^{m}}\\
     \Leftrightarrow&
     \sum_{j=0}^{m-1} \lambda_{j} x^{j} \alpha_{i}^{} \equiv (\alpha_{i}^{}-a_{i,0})/x &\pmod{x^{m}}\\
     \Leftrightarrow&
     \sum_{j=0}^{m-1} \lambda_{j} x^{j+1} \alpha_{i}^{} \equiv \alpha_{i}^{}-a_{i,0} &\pmod{x^{m+1}}\\
     \Leftrightarrow&
   \Lambda^{(m)}(x) \alpha_{i}^{} \equiv a_{i,0} &\pmod{x^{m+1}}, \\
  \end{align*}
   for some nonzero polynomial $\Lambda^{(m)}(x)$. 
   But this can only occur if all $\alpha_{i}^{(m+1)}$ are constant multiples of each other. Requesting the $\alpha_{i}$ to be linearly independent over $\FF_{q}$, which is the only interesting case by Theorem~\ref{udt}, rules this out for large enough $m$, i.e., at least for large $m$ the condition \eqref{stern} of the corollary is satisfied. But this suffices, since the proof of Theorem \ref{thm9} actually shows how to construct $C^{(m+1)}$ uniquely, given
 $C^{(m,m+1)}$. Furthermore, any other $C'\in \FF_{q}^{\NN\times\NN}$ generating the same sequence would have to fulfill ${C'}_{i}^{(m)}={C}_{i}^{(m)}P^{(m)},\,i=1,\dots,s$, with $P^{(m)}$ being a NUT matrix.
 But this can never happen, if the $C_{i}$ are also NUT matrices and the $C'_{i}$ are Hankel matrices. Hence
 Larcher-Niederreiter sequences and hyperplane sequences have at most trivial common 
generator matrices and are distinct as constructions.

\section{An Existence Result}\label{secExRes}

In this section we prove an existence result which is based on Corollary~\ref{llnsc1}. In the following let $$\widetilde{Y}^s:=\{\bsalpha=(\alpha_1,\ldots,\alpha_s) \in Y^s\ : \ \alpha_1=1\}.$$

\begin{theorem}\label{thmex}
Let $\bsbeta=(\beta_m)_{m \ge 1}$ be such that $\sum_{m \ge 1}\beta_m< \infty$. Then there exists an integer $L=L(s,q,\bsbeta)$ and an $\bsalpha \in \widetilde{Y}^s$ such that the associated hyperplane sequence $\mathcal S_{\bsalpha}$ is a digital $(\Tfett,s)$-sequence over $\FF_q$ with $$\Tfett(m) \le \log_q \frac{m^{s-1}}{\beta_m} +L \ \ \mbox{ for all }\ \ m \ge 1.$$  
\end{theorem}

For example if we choose $\bsbeta$ such that $\beta_m=\frac{1}{m(\log m)^2}$ for $m \ge 2$, then we obtain the following corollary:
\begin{corollary}\label{co2}
There exists an integer $L$ and an $\bsalpha \in \widetilde{Y}^s$ such that the associated hyperplane sequence $\mathcal S_{\bsalpha}$ is a digital $(\Tfett,s)$-sequence over $\FF_q$ with $$\Tfett(m) \le s \log_q m + 2 \log_q \log m  +L \ \ \mbox{ for all }\ \ m \ge 2.$$ 
\end{corollary}

\begin{corollary}\label{co3}
There exists a $\bsalpha \in \widetilde{Y}^s$ such that the star discrepancy of the associated hyperplane sequence $\mathcal S_{\bsalpha}$ satisfies $$D_N^{\ast}(\mathcal S_{\bsalpha}) =O\left(\frac{(\log N)^{2s}(\log \log N)^{2}}{N}\right)\ \ \mbox{ for all } N \ge 3,$$ where the implied constant depends only on $s$ and $q$. 
\end{corollary}

\begin{proof}
Combine Corollary~\ref{co2} with the star discrepancy bound given in \cite[Eq. (5.5), p. 196]{DP2010}. 
\end{proof}

For the proof of Theorem~ \ref{thmex} we need some notation: Let $Y_m=\{\alpha^{(m)}\, : \, \alpha \in Y\}$. Note that $|Y_m|=q^{m-1}(q-1)$. We define now a probability measure over the set $Y$. We do this in a way such that the measure of corresponding vectors in $Y_m$ is equiprobable. For $m \in \NN$ let $\mu_m$ be the equiprobable measure on the set $Y_m$. We say a subset $A$ of $Y$ is of finite type, if there exists an integer $m=m(A) \in \NN$ and a subset $A'$ of $Y_m$ such that $$A=\{\alpha \in Y\, : \, \alpha^{(m)} \in A'\}.$$ The measure of the finite type subset $A$ is then defined as $$\mu(A)=\mu_m(A')=\frac{|A'|}{q^{m-1}(q-1)}.$$  
(Of course, the number $m=m(A)$ is not uniquely defined by $A$. However, it is easy to see that the definition of $\mu$ does not depend on the specific choice of $m$.)  

For $s \in \NN$ let $\mu^{(s)}$ denote the complete product measure on $Y^s$ induced by $\mu$. 

Now we are ready to prove Theorem~\ref{thmex}.

\begin{proof}
Let $$\mathcal U_m(\beta)=\left\{\bsalpha \in \widetilde{Y}^s\, : \, \rho(\bsalpha^{(m)}) \ge \left\lfloor m +\log_q \beta- (s- 1) (\log_q m -1)+ \log _q \frac{(s-1)!}{q^{s-1}}\right\rfloor\right\}.$$ Then it follows from 
Corollary~\ref{llnsc1} that there exists some $m_0=m_0(s,q)$ such that for all $0 < \beta \le 1$ and all $m \ge m_0$ we have $$\mu_s(\mathcal U_m(\beta)) > 1-\beta.$$

Let $\overline{\mathcal U}_m(\beta)$ be the complement of $\mathcal U_m(\beta)$. Let $\bsbeta=(\beta_m)_{m \ge 1}$ be such that $\sum_{m \ge 1}\beta_m< \infty$. Let $n_0=n_0(s,q,\bsbeta)$ be such that $n_0 \ge m_0$ and $\sum_{m \ge n_0}\beta_m\le 1$. Then we have
\begin{eqnarray*}
\mu_s\left(\bigcap_{m \ge n_0}\mathcal U_m(\beta_m)\right) & = & 1-\mu_s\left(\bigcup_{m \ge n_0}\overline{\mathcal U}_m(\beta_m)\right) \\
& \ge & 1-\sum_{m \ge n_0} \mu_s(\overline{\mathcal U}_m(\beta_m))\\
& > & 1-\sum_{m \ge n_0} \beta_m\\
& \ge & 0.
\end{eqnarray*}

Hence there exists a $\bsalpha \in \widetilde{Y}^s$ such that $$\rho(\bsalpha^{(m)}) \ge \left\lfloor m +\log_q\beta_m- (s- 1) (\log_q m -1)+ \log _q \frac{(s-1)!}{q^{s-1}}\right\rfloor\ \ \mbox{ for all }\ \ m \ge n_0.$$ From this it follows that there exists some integer $L=L(s,q,\bsbeta)$  and there exists a $\bsalpha \in \widetilde{Y}^s$ such that $$\rho(\bsalpha^{(m)}) \ge \left\lfloor m +\log_q\beta_m- (s- 1)\log_q m -L\right\rfloor \ \ \mbox{ for all }\ \ m \ge 1.$$

Now the result follows from Corollary~\ref{co1}.
\end{proof}

\begin{small}
\noindent\textbf{Authors' address:}\\
\noindent Friedrich Pillichshammer, Gottlieb Pirsic\\
Institut f\"{u}r Finanzmathematik, Universit\"{a}t Linz, Altenbergerstr.~69, 4040 Linz, Austria\\
E-mail: 
\texttt{friedrich.pillichshammer@jku.at}, \texttt{gpisic@gmail.com} 
\end{small}

\end{document}